\documentclass[12pt]{article}
\usepackage{amsmath}
\usepackage{mathrsfs} 
\usepackage{amsthm}
\usepackage{amsfonts}
\usepackage{slashed}
\usepackage{chngcntr}
\usepackage{amssymb}
\usepackage{latexsym}
\usepackage{apptools}
\usepackage{tikz} 
\usepackage{esint} 
\usepackage{tikz}
\usepackage{tikz-cd}

\usepackage[matrix,tips,graph,curve]{xy}



\makeatletter 
\@addtoreset{equation}{section}
\makeatother

\theoremstyle{plain}
\newtheorem{theorem}[equation]{Theorem}
\newtheorem{corollary}[equation]{Corollary}
\newtheorem{lemma}[equation]{Lemma}
\newtheorem{proposition}[equation]{Proposition}

\theoremstyle{definition}
\newtheorem{definition}[equation]{Definition}

\newtheorem{remark}[equation]{Remark}

\newtheorem{example}[equation]{Example}

\newcommand{\IC}{\mathbb{C}}

\newcommand{\IP}{\mathbb{P}}

\newcommand{\IR}{\mathbb{R}}

\newcommand{\IZ}{\mathbb{Z}}

\renewcommand{\deg}{\mathrm{deg}}

\renewcommand\Re{\mathrm Re}

\def\d/{/\mspace{-6.0mu}/}

\let\oldphi\phi \let\phi\varphi \let\varphi\oldphi

\begin{document}
\title{On $\mathbb{Z}_2$ harmonic functions on $\IR^2$ and the polynomial Pell equation}

\date{}
\author{Weifeng Sun}
\maketitle

There have been many studies on ``$\mathbb{Z}_2$" harmonic functions, differential forms, or spinors recently, for example \cite{Taubes2021ExamplesThree}\cite{CliffordTopologicalS2}\cite{Donaldson2021DeformationsFunctions}\cite{SiqiHe2022ExistenceSymmetry}. It is then natural to ask what the most simple case is like: $\mathbb{Z}_2$ harmonic functions on $\mathbb{R}^2$ with point singularities? This paper studies such $\mathbb{Z}_2$ harmonic functions and finds its unexpected relationship with the polynomial Pell equation.\\

\textbf{Remark:} Unlike the more ambitious ongoing program  of studying $\mathbb{Z}_2$ harmonic functions $f$ such that $|df| = 0 $ on the singular locus $K$ (see, for example, \cite{Donaldson2021DeformationsFunctions}\cite{SiqiHe2022ExistenceSymmetry}), this paper only requires $|df|$ to have locally finite $L^2$ norms near $K$, which means, $|df|$ is allowed to be unbounded near $K$. The readers should not confuse what we study here with the other ongoing program.\\

\section*{Summary of the Results}

Let $X$ be $\IR^2 \backslash K$, where $K$ is a collection of an even number of isolated points.  We will also identify $\IR^2$ with $\IC$ and suppose $K = \{z_1, z_2, \cdots, z_{2k}\} \subset \IC$. Then there is a unique real line bundle $\mathcal{L}$ in $X$ that has monodromy $-1$ around each point in $K$. An $\mathcal{L}$-valued function/differential form is also called a $\mathbb{Z}_2$ function/differential form.\\

Locally, a $\mathbb{Z}_2$ function $f$ can be regarded as an ordinary function under a trivialization of $\mathcal{L}$. This function changes sign in a different trivialization, which is why it is called a ``$\mathbb{Z}_2$" function. In particular, the derivatives of $f$ make sense if $f$ is differentiable (which are also $\mathbb{Z}_2$ functions).

\begin{definition}
A $\mathbb{Z}_2$ harmonic function $f$ is a $\mathbb{Z}_2$ function whose Laplacian is $0$. That is to say, under any local trivialization,
$$\Delta f = (\partial_x^2 + \partial_y^2) f = 0.$$

In addition, in this paper, we assume that $|\nabla f|$ has a bounded $L^2$ norm on any bounded open subset of $X$ (including an open neighborhood of $K$).
\end{definition}

Choose a large disk $B \subset \mathbb{C}$ that contains $K$ as a subset. Then the line bundle $\mathcal{L}$ is trivial on $\mathbb{C} \backslash B$. For convenience, we fix a trivialization $\iota$ of $\mathcal{L}$ in $\mathbb{C} \backslash B$. Under this trivialization, a $\mathbb{Z}_2$ function $f$ is an ordinary real-valued function in $\mathbb{C} \backslash B$.\\

Here are the main results of this paper (proofs will be deferred to later sections).

\begin{theorem}\label{theorem of two dimensional correspondence}
    Suppose $$D(z) = \prod\limits_{i=1}^{2k}(z - z_i).$$
Then there is a natural 1-1 correspondence between any two of the following three objects:

\begin{itemize}
    \item A $\mathbb{Z}_2$ harmonic function $f$ on $X$.
    \item An admissible entire function $u(z)$ on $\mathbb{C}$, where \textbf{admissible} means for any $j = 2, 3, \cdots, 2k$, 
$$\text{Re} (\int_{z_1}^{z_j} \dfrac{ u(z)}{\sqrt{D(z)}} dz) = 0, $$
and $\text{Re}$ means the real part.
   \item An equivalence class $[h]$ of harmonic functions in $\mathbb{C} \backslash \{O\}$, where $O$ is the origin, given by:
    $$h_1 \sim h_2 ~~~~ \text{if and only if} ~~~ h_1 - h_2 \rightarrow C ~~ \text{when} ~~ |z| \rightarrow +\infty, $$
    where $C$ is a constant.
\end{itemize}

To be more precise, their relationships are given by
$$f(z) = \text{Re} (\int_{z_1}^z \dfrac{u(z)}{\sqrt{D(z)}} dz) ~~\text{and} ~~ \lim\limits_{|z| \rightarrow +\infty} (f(z) - h(z)) = C ~~~ \text{for some constant $C$.} $$
\end{theorem}

\begin{remark}
Strictly speaking, the integral $\displaystyle \int_{z_1}^{z_j} \dfrac{u(z)}{\sqrt{D(z)}} dz$ may depend on a path chosen from $z_1$ to $z_j$ in $K$ and a sign chosen for $\sqrt{D(z)}$. But for any paths chosen from $z_1$ to each $z_j$ and any sign chosen for $\sqrt{D(z)}$ the admissible conditions are equivalent.
Moreover, the admissible condition is equivalent to saying that the integral $ \displaystyle \text{Re} (\int_{z_1}^z \dfrac{u(z)}{\sqrt{D(z)}} dz)$ is well defined up to a sign (or, to be more precise, as a section of the line bundle $\mathcal{L}$). Details can be found in section \ref{Section: Properties} remark \ref{Remark well define admissible}.\\
\end{remark}

There is an analog correspondence in the higher-dimensional case but even simpler: We assume that the ambient space is $X = \mathbb{R}^n \backslash K$, where $n \geq 3$ and $K$ is a compactly smoothly embedded sub-manifold of dimension $n-2$ in $\mathbb{R}^n$. Let $\mathcal{L}$ be the unique real line bundle on $X$ with monodromy $-1$ around $K$. We can define $\mathbb{Z}_2$ harmonic functions $f$ on $X$ with monodromy $-1$ around $K$ in the same way. Then
\begin{theorem}\label{theorem of higher dimensional analog}
There is a bijection between:

\begin{itemize}
    \item $\mathbb{Z}_2$ harmonic functions $f$ on $\IR^n \backslash K$.
    \item Harmonic functions $h$ on $\IR^n$.
    \end{itemize}
The bijection is given by $$f - h \rightarrow 0 ~~~~~~ \text{when}~~~~ |\mathbf{x}| \rightarrow +\infty, $$
where $\mathbf{X}$ is the coordinates of $\mathbb{R}^n$.
\end{theorem}

\begin{remark}
    Soon after the author wrote the first draft of this paper, professor Mazzeo told the author a shorter argument to prove the theorem \ref{theorem of higher dimensional analog}, which indicates that this result might not be unexpected for experts. (Note that the $n = 2$ case is different.) But the author chose not to change the original argument here for two reasons: 1. Mazzeo's argument relies on more background knowledge on geometrical analysis, while the current one, though more tedious in details, is more elementary in a certain sense. 2. The author believes that the original argument could also potentially be illuminating in other studies (for example, the study of more complicated $\mathbb{Z}_2$ harmonic objects).
\end{remark}

Nevertheless, the following corollary of theorem \ref{theorem of higher dimensional analog} is new and partially answers a question asked by Mazzeo which roughly says:\\

     \textit{Suppose that there is a unit circle $C$ in the Euclidean space $\mathbb{R}^3$, is there a $\mathbb{Z}_2$ harmonic function $f$ on $\mathbb{R}^3 \backslash C$ with monodromy $-1$ around $C$ such that $|f|$ extends to a continuous function throughout $\mathbb{R}^3$? Will it be useful in gluing arguments in gauge theory?}\\

\begin{corollary}\label{Corollary of higher dimensional case}
Suppose $C$ is a unit circle embedded in the Euclidean space $\mathbb{R}^3$. Let $\rho$ be the distance to the circle $C$. Then for any positive integer $k$, there is a $\mathbb{Z}_2$ harmonic function $f$ on $\mathbb{R}^3 \backslash C$ with monodromy $-1$ around $C$ such that $|f| = O(\rho^k)$ as $\rho \rightarrow 0$.    
\end{corollary}

The proof is given in section \ref{Section: Properties}.\\

After the digression on the higher-dimensional case, we come back to the two-dimensional case. Suppose $K$ and $D(z)$ are given. Consider the following harmonic function in $\mathbb{C} \backslash \{O\}$: $$h(z) = \ln|z|.$$ 
By theorem \ref{theorem of two dimensional correspondence}, there is a unique $\mathbb{Z}_2$ harmonic function $f(z)$ (denoted $G_D(z)$, which depends on $D(z)$) and a unique admissible entire function $U_D(z)$ that corresponds to it. 
    Then they satisfy the following properties:
\begin{theorem}\label{theorem of Pell} The triple $(G_D(z), U_D(z), \ln |z|)$ is ``minimal" in the following sense:
\begin{itemize}
    \item Any $\mathbb{Z}_2$ harmonic function $f$ such that $f = o(|z|)$ as $|z| \rightarrow +\infty$ is a scalar multiplication of $G_D(z)$ by a real-valued constant.
    \item The zero locus of $G_D(z)$: $Z(G_D(z)) := \{z \in \mathbb{C}~|~ G_D(z) = 0\}$ is compact. Moreover, any $\mathbb{Z}_2$ harmonic function that satisfies this property is a multiplication of $G_D(z)$ by a real-valued constant.
    \item Any admissible entire function $u(z)$ in $\mathbb{C}$ such that $u(z) = o(|z|^k)$ as $|z| \rightarrow +\infty$ is a scalar multiplication of $U_D(z)$ by a real-valued constant.
\end{itemize}
\end{theorem}
Moreover, we have the following mysterious relationship between $U_D(z)$ and the polynomial Pell equation:
\begin{theorem} \label{theorem of Pell property}  If $(p(z), q(z))$ is a pair of polynomials that satisfies the following polynomial Pell equation:
    $$p(z)^2 - D(z)q(z)^2 = 1, $$
    then $\dfrac{p'(z)}{q(z)}$ is a multiplication of $U_D(z)$ up to a multiplication by a real valued constant.\\
    
    Moreover, a pair $(p(z), q(z))$ satisfies the polynomial Pell equation if any only if the following algebraic expression of the Abel integration holds:
    $$\int \dfrac{U_D(z)}{\sqrt{
    D(z)}} dz = C \ln(p(z) + q(z)\sqrt{D(z)}), ~~~ \text{where $C$ is a real-valued constant}.  $$
\end{theorem}

\begin{remark}~
    \begin{itemize}
       \item Not all $D(z)$ whose corresponding polynomial Pell equation has a solution. If it has a solution, then $D(z)$ is called Pellian.
       \item We may easily extend the definition of $U_D(z)$ to the case where $D(z)$ has roots of multiplicity greater than $1$. See definition \ref{Definition of U}.
        \item When $D(z)$ is Pellian, the existence of such a polynomial $U_D(z)$ that satisfies the properties in theorem \ref{theorem of Pell property} is known. (See \cite{Dubickas2004TheEquation} and \cite{Nathanson1976PolynomialEquations} for example.) However, to the author's knowledge, its natural relationship with $\mathbb{Z}_2$ harmonic functions is a new result. 
        \item Our definition of $U_D(z)$ (through the correspondence given by theorem \ref{theorem of two dimensional correspondence}), which still makes sense even when $D(z)$ is non-Pellian, to the author's knowledge, is also new.
    \end{itemize}
\end{remark} 

    With the help of our generalized definition of $U_D(z)$ (which makes sense even if $P(z)$ is non-Pellian), there is a natural practical method to check whether any given $D(z)$ is Pellian or not. (Roughly speaking, for any given $D(z)$, Pellian or not, one may find $U_D(z)$ first and then check whether the Abel integration is well-defined or not.)\\
    
    Details of this method are described in section \ref{Section: Pell}. The author does not know whether this is the only known method or not. (The author suspects that it is not, since polynomial Pell equations have been a well-studied subject for a long time, and the author is not an expert in that.) But the approach through $\mathbb{Z}_2$ harmonic functions is new.\\

There are also some miscellaneous studies on the zero locus of $\mathbb{Z}_2$ harmonic functions (and their relationship with $u(z)$) which are stated in section \ref{Section: zero locus} but omitted in the introduction.\\

\textbf{Acknowledgment:} I want to thank my Ph.D. advisor, Prof. Taubes, for his endless support, inspiration, and encouragement. Several arguments used in the paper are due to Taubes (which will be mentioned precisely), and more ideas were inspired by him. This paper cannot be finished without his help. I also want to thank Zhenkun Li, who helped me calculate the first homology group of the double branched covering over $\IC$ and Prof. Mazzeo, who encourages and supports me all the time. And I want to thank Boyu Zhang, Kai Xu, Ziquan Yang, Yongquan Zhang, Siqi He for helpful discussions with them.

\section{Properties of $\mathbb{Z}_2$ harmonic functions}\label{Section: Properties}

This section mainly proves theorem \ref{theorem of two dimensional correspondence}, theorem \ref{theorem of higher dimensional analog}, and corollary \ref{Corollary of higher dimensional case} in the introduction. 

\begin{lemma}\label{local neibourhood lemma}
If $f$ is a $\mathbb{Z}_2$ harmonic function on $X$, then for each $z_i \in K$, there exists an open neighborhood $U_i$ of $z_i$, and a $\mathbb{Z}_2$ holomorphic function $\phi_i(z)$ on $U_i \backslash \{z_i\}$, such that
$$f = \text{Re}(\phi_i(z)), ~~~~ \text{for}~~z \in U_i\backslash\{z_i\} $$
where Re is the real part and
$$\phi_i(z) = a_i(z - z_i)^{k_i + \frac{1}{2}} + O(|z-z_i|^{k_i + \frac{3}{2}}).$$

We call the half integer $k_i + \dfrac{1}{2}$ the degree of $f$ at $z = z_i$, or $\deg_{z_i} f(z)$.
\end{lemma}

\begin{proof}
Suppose $z - z_i = \rho e^{i\theta}$, where $(\rho, \theta)$ is the polar coordinates.
Take the Fourier series of $f$ with respect to $\theta$ in a small neighborhood of $z_i$. Locally $f$ can be written as
$$f = \text{Re}(\sum\limits_{m=0}^{+\infty} a_m(\rho)e^{\frac{(2m+1)i\theta}{2}}),$$
where $a_m(\rho)$  are complex-valued functions in $\rho$.\\

(The Fourier series exists because $f$ is an ordinary smooth function on the double cover of a neighborhood of $P$ with the exclusion of $P$ itself, and one can do the Fourier series there. The integer terms disappear because $f$ has monodromy $-1$.)\\

Based on the Laplacian equation,
$$\dfrac{d^2a_m(\rho)}{d\rho^2} + \dfrac{1}{\rho} \dfrac{da_m(\rho)}{d\rho} - \dfrac{(2m+1)^2}{4\rho^2}a_m(\rho) = 0.$$

So $a_m(\rho) = A_m \rho^{\frac{2m+1}{2}} + \tilde{A}_m \rho^{-\frac{2m+1}{2}}$, where $A_m$ and $\tilde{A}_m$ are complex-valued constants. Since $|\nabla f|$ has locally finite $L^2$ norm, $\tilde{A}_m$ must be $0$. So

$$f = \text{Re}( \sum\limits_{m=0}^{+\infty} A_m\rho^{\frac{2m+1}{2}}e^{\frac{(2m+1)i\theta}{2}})= \text{Re} (\sum\limits_{n=0}^{+\infty}A_m(z-z_i)^{\frac{2m+1}{2}}).$$

Let $\phi_i(z) =\sum\limits_{n=0}^{+\infty}A_m(z-z_i)^{\frac{2m+1}{2}} $ and let $m = k_i$ be the smallest integer such that $A_m \neq 0.$ Then the proof is finished.
\end{proof}

Here are two immediate corollaries.

\begin{corollary}\label{corollary of |f| = 0 on K}
If $f$ is a $\mathbb{Z}_2$ harmonic function, then $|f|$ can be extended to $\IR^2$ continuously by setting $|f| = 0$ on $K$.
\end{corollary}

\begin{corollary}\label{local neibourhood corollary}
The expression $\dfrac{\phi_i(z)}{(z - z_i)^{k_i + \frac{1}{2}}}$ represents a holomorphic function in $U_i$ up to a sign. Suppose $$D(z) = \prod\limits_{i=1}^{2k}(z - z_i),$$
then $\dfrac{\phi_i(z)}{\sqrt{D(z)}}$ represents a holomorphic function in $U_i$ up to a sign. When $k_i \geq 1$,  $\dfrac{\phi_i(z)}{\sqrt{D(z)}}$ has a zero point of degree $k_i$ at $z = z_i$.
\end{corollary}

Note that in general, $\dfrac{\phi_i(z)}{\sqrt{D(z)}}$ is not a global holomorphic function.

\begin{theorem}\label{entire function}
There exists an entire function $u(z)$ on $\IC$ such that
$$df = \text{Re} (\dfrac{u(z)dz}{\sqrt{D(z)}}). $$
Moreover, if $k_i \geq 1$, then $u(z)$ has a zero point at $z = z_i$ of degree $k_i$.
\end{theorem}

\begin{proof}
Both $df$ and $*df$ are $\mathbb{Z}_2$ 1-forms on $\IR^2 \backslash K$. And 
$$df + i*df = (f_x dx + f_y dy + if_x dy - i f_y dx)  $$
$$= (f_x - if_y)(dx + idy) = (f_x - if_y)dz. $$
Note that 
$$ \bar{\partial}(f_x - if_y) = \dfrac{1}{2}(\partial_x + i\partial_y)(f_x - if_y) = \dfrac{1}{2}(f_{xx} + f_{yy}) = 0. $$

So $f_x - i f_y$ is a $\mathbb{Z}_2$ holomorphic function on $\IR^2 \backslash K$ with monodromy $-1$ at $K$.\\

Let $u(z) = \sqrt{D(z)} (f_x - if_y)$. Then $u(z)$ is an ordinary holomorphic function in $\IR^2 \backslash K$. (Strictly speaking, $u(z)$ is determined only up to a sign here. But if we fix a sign convention for $\sqrt{D(z)}$ as a section of $\mathcal{L}$, then $u(z)$ is unique. We will always assume that this is the case.)\\

Let $\phi_i(z)$ be the function defined in lemma \ref{local neibourhood lemma} and corollary \ref{local neibourhood corollary}, then $\dfrac{\phi_i(z)}{\sqrt{D(z)}}$ is holomorphic in a neighborhood of $z_i$ and 
$$f_x - if_y = \phi_i'(z) ~~~~ \text{in a neighborhood of $z_i$}. $$

Thus $u(z) = \sqrt{D(z)} \phi_i'(z) =  (\dfrac{\phi_i(z)}{\sqrt{D(z)}})' D(z) + \dfrac{1}{2} (\dfrac{\phi_i(z)}{\sqrt{D(z)}}) D'(z)$ is also holomorphic in a neighborhood of $z_i$. And $u(z)$ is an entire function in $\IC$.
\end{proof}

\begin{remark}\label{Remark well define admissible}
If $f$ is a $\mathbb{Z}_2$ harmonic function, then $f(z_1) = 0$. So we can always write $f$ as $$f(z) = \text{Re} (\int_{z_1}^z \dfrac{u(z)}{\sqrt{D(z)}} dz)$$
for some entire function $u(z)$ in $\IC$.
However, clearly only admissible entire functions $u(z)$ in $\IC$ make the above expression well-defined $\mathbb{Z}_2$ harmonic functions. 
\end{remark}

In general, an integral such as
$\displaystyle \int_{z_1}^z \dfrac{u(z)}{\sqrt{D(z)}}dz$ depends on the path chosen in $\Sigma$ from a lifting of $z_1$ to a lifting of $z$, where $\Sigma$ is the surface described by $\{ (w, z) \in \mathbb{C}^2~|~w^2 = D(z)\}$, regarded as a double branched cover of $\mathbb{C}$.\\

The first homology group of $\Sigma$ is generated by the loops $C_2, \cdots, C_{2n}$ as follows:\\

Choose any path $\gamma_j$ from $z_1$ to $z_j$ in $\IC$ that does not meet any other point in $K$. Then $C_j$ is defined to be going along $\gamma_j$ from $z_1$ to $z_j$ in any branch in $\Sigma$, and go back along $\gamma_j$ from $z_j$ to $z_1$ but in another branch.\\

In fact,
one can visualize the double-branched cover from $\Sigma$ to $\mathbb{C}$ as follows:\\
 
\includegraphics[scale = 0.08]{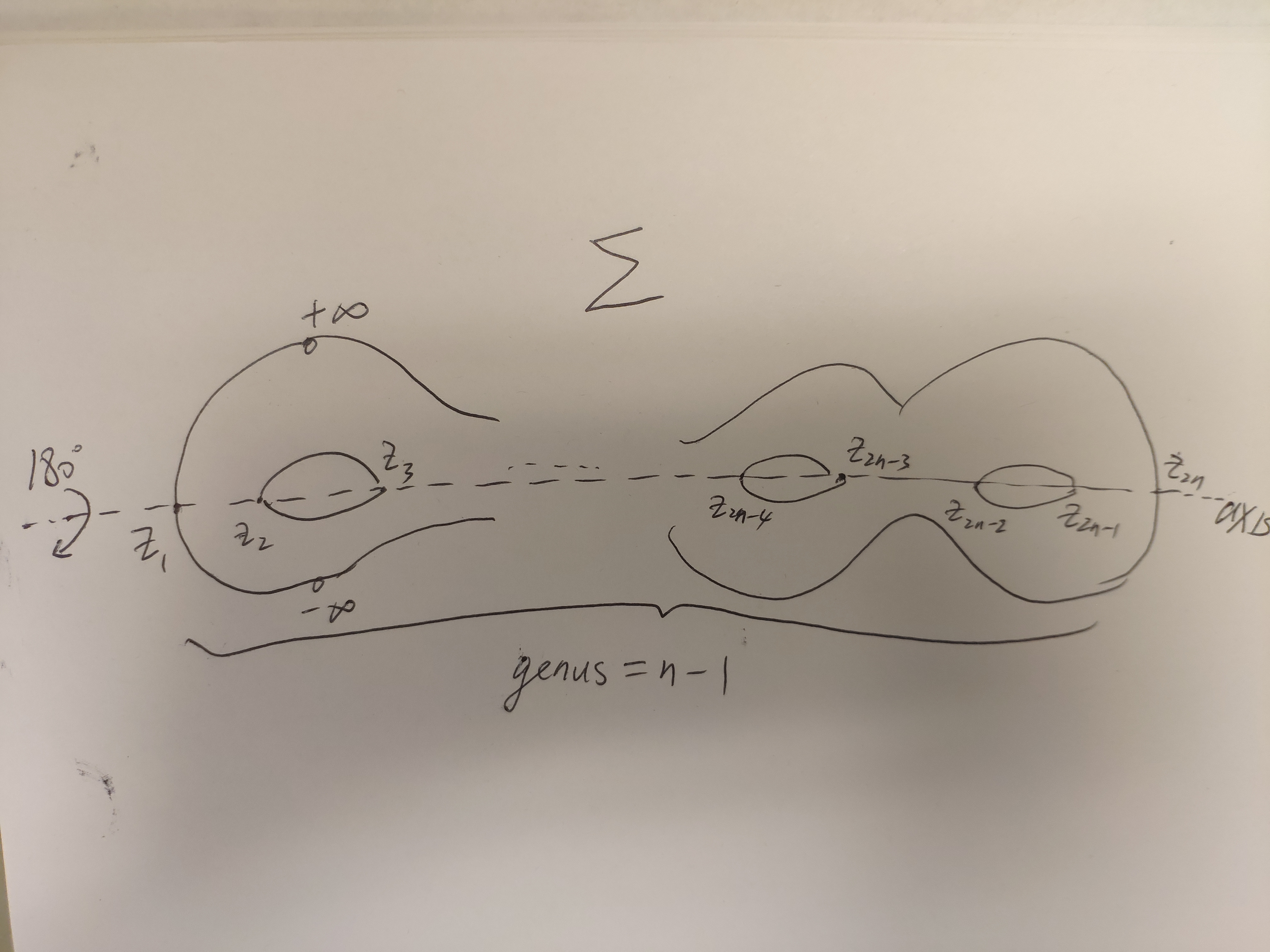}

In this graph, let $\iota$ be the rotation of $\Sigma$ along the axis by $180^o$. Then the branched double cover from $\Sigma$ to $\IC$ is given by the $\Sigma$ quotient $\iota$. The intersection of $\Sigma$ with the axis is exactly the branching points $z_1, z_2, \cdots, z_{2n}$. The two punctured points $+\infty$ and $-\infty$ are the two points above infinity.\\

\begin{proposition}
    
The definition of admissible in thoerem \ref{theorem of two dimensional correspondence} is well-defined.
\end{proposition}

\begin{proof}
Clearly if one chooses a different sign of $\sqrt{D(z)}$,
$$ \text{Re}(\int_{z_1}^{z_j} \dfrac{u(z)}{\sqrt{D(z)}})dz$$
only changes sign.\\

Now suppose for some chosen paths $\gamma_j$ from $z_1$ to $z_j$, and for each $j$,
$$\text{Re}(\int_{z_1}^{z_j} \dfrac{u(z)}{\sqrt{D(z)}}dz) = \text{Re}(\int_{\gamma_j} \dfrac{u(z)}{\sqrt{D(z)}}dz) = 0.$$

Let $C_j$ be the loop on $\Sigma$ from $z_1$ to $z_j$ along $\gamma_j$ on the one branch, and then coming back from $z_j$ to $z_1$ along $\gamma_j$ but on the other branch. Then
$$\text{Re}(\int_{C_j} \dfrac{u(z)}{\sqrt{D(z)}}dz) = 2 ~ \text{Re}(\int_{\gamma_j} \dfrac{u(z)}{\sqrt{D(z)}} dz) = 0. $$

But since $C_2, C_3, \cdots, C_{2n}$ are generators of $H_1(\Sigma,\IZ)$, we know the integral
$$\int_{\gamma}\dfrac{u(z)}{\sqrt{D(z)}}dz $$
is purely imaginary along any closed loop $\gamma$ in $\Sigma$. So, the integral
$$\int_{z_1}^{z_j} \dfrac{u(z)}{\sqrt{D(z)}}dz $$
has a well-defined real part which does not depend on the path chosen from $z_1$ to $z_j$. In particular, the condition that ``their real parts are all zero" does not depend on the specific paths that are used to define them and is well defined.

\end{proof}

\begin{theorem}\label{theorem of equivalence 1}
There is a bijection between
\begin{itemize}
    \item Admissible entire functions $u(z)$ on $\IC$;
    \item $\mathbb{Z}_2$ harmonic functions $f$ on $X$.
\end{itemize}
The bijection is given by 
$$f(z) =\text{Re} (\int_{z_1}^{z} \dfrac{u(z)}{\sqrt{D(z)}} dz). ~~~ (*)$$

\end{theorem}

\begin{proof}
~

\begin{itemize}
    \item $f(z) \Rightarrow u(z)$:   From theorem \ref{entire function}, we know the entire function $u(z)$ exists. By corollary \ref{corollary of |f| = 0 on K},  
 $$0 = f(z_j) - f(z_1) = \text{Re} (\int_{z_1}^{z_j} \dfrac{u(z)}{\sqrt{D(z)}} dz).$$
\end{itemize}  

 \begin{itemize}
    \item $u(z) \Rightarrow f(z)$:
    As shown in the Appendix A, if $u(z)$ is admissible, then the real part of the integration $$\int_{z_1}^{z} \dfrac{u(z)}{\sqrt{(D(z))}}dz$$ doesn't depend on the path chosen from $z_1$ to $z$ except its sign (which is also fixed is a sign convention for $\sqrt{D(z)}$ is chosen). So its real part is a well-defined harmonic function.  Call this function $f$. We have $f(z_j) = 0$ for any $ j = 1, 2, \cdots, 2k$.  In order for $f$ to be a $\mathbb{Z}_2$ function, we need to check that the monodromy of $f$ around each point $z = z_j$ is $-1$.\\
    
    Near each $z = z_i$, the Taylor series looks like
    $$ \dfrac{u(z)}{\sqrt{D(z)}} = \sum\limits_{k=0}^{+\infty} A_k (z-z_i)^{k-\frac{1}{2}}.$$
    So
    $$f = \text{Re}(\sum\limits_{k=0}^{+\infty}A_k(k + \dfrac{1}{2})^{-1}(z - z_i)^{k+\frac{1}{2}}), $$
    which indeed has the correct monodromy around $z = z_i$.
\end{itemize}  
\end{proof}

Since $u(z)$ is uniquely determined by $f$, sometimes it is convenient to use the notation $u_{f}(z)$ to emphasize its dependence on $f$. However, if there is no ambiguity, the subscript $f$ can be omitted.\\

Suppose $(\rho, \theta)$ are the polar coordinates centered at the origin. Suppose $B_R$ is a ball with a large radius $R$ centered at the origin such that $K \subset B_R$. On $\IR^2 \backslash B_R$. Using Fourier series, 
\begin{equation}\label{Fourier series of f at infinity}
    f =  (c_0 \ln \rho + \text{Re}(\sum\limits_{n=1}^{+\infty} c_n \rho^n e^{in\theta})) + (b_0 + \text{Re}(\sum\limits_{n=1}^{+\infty} \dfrac{b_n}{\rho^n} e^{in \theta})) = U + B, 
\end{equation} 
where $b_n, c_n$ are complex numbers except that $b_0, c_0$ must be real. Capital letters $U$ and $B$, representing the first and second parentheses, respectively, are the ``unbounded" and ``bounded" parts at infinity. We may also call them $U(f)$ and $B(f)$ respectively to emphasize their dependency on $f$. Note that $U(f)$ is actually a harmonic function on the whole $\IR^2\backslash \{O\}$, where $O$ is the origin. (But $B(f)$ is not.)\\

\begin{theorem}\label{uniqueness theorem}
Suppose that there are two $\mathbb{Z}_2$ harmonic functions $f_1, f_2$ in $X$ such that $U(f_1) = U(f_2)$, then $f_1 = f_2$. In particular, if $U(f) = 0$, then $f = 0$.

\end{theorem}

\begin{proof} Since $f = f_1 - f_2$ is a $\mathbb{Z}_2$ harmonic function with $U(f) = 0$. So it suffices to prove $f = 0$.\\

Assume when $\rho > R$,
$$ f = b_0 + \text{Re}(\sum\limits_{n=1}^{+\infty} \dfrac{b_n}{\rho^n} e^{in \theta}).$$

We may assume $b_0 \geq 0$, since otherwise we can change the trivialization of $\mathcal{L}$ and flip the sign of $f$. Then $b_0 - |f|$ is a super-harmonic function on the entire $\IR^2$ which approaches $0$ at infinity. By maximum principle, it cannot be negative. So in particular, when $\rho > R$, $f - a_0 = \text{Re}(\sum\limits_{n=1}^{+\infty} \dfrac{a_n}{\rho^n}e^{in\theta}) \geq 0$.

Suppose $f \neq 0$. Then let $k$ is the smallest positive integer such that $a_k \neq 0$. Then $$ 0 \leq f - a_0 = \text{Re} (\dfrac{a_k}{\rho^k}e^{ik\theta}) + O(\rho^{-(k+1)}),$$
which is impossible when $\rho$ is large enough.

\end{proof}

\begin{corollary}\label{unique corollary with bounded zero locus}
If the zero locus of $f$ is bounded, then $U(f) = c_0 \ln \rho$ for some constant $c_0 \neq 0$. In addition, such $f$ (if exists) is unique up to a scalar multiplication by a constant.
\end{corollary}
\begin{proof}
Suppose the zero locus of $f$ is bounded and in general, 
$$U(f) = c_0 \ln \rho + \text{Re}(\sum\limits_{n=1}^{+\infty} c_n \rho^n e^{in\theta}).$$
We may assume $f > 0$ when $\rho > R$ for some large enough $R$.\\

Then $U(f) - c_0\ln \rho$ is a harmonic function on the entire $\IR^2$ which is bounded below by $-c_0 \ln \rho$ when $\rho > R$. It is well known that such a harmonic function with a lower bound must be a constant. But this constant can only be $0$ based on the expression of $U(f)$. So, the only possibility is $U(f) = c_0 \ln \rho$.
The uniqueness of $f$ up to a scalar multiplication follows directly from theorem \ref{uniqueness theorem}.
\end{proof}

One may wonder for which choice of $U(f)$ the $\mathbb{Z}_2$ harmonic function exists. In fact, we have the following theorem:

\begin{theorem}\label{theorem of equivalence 2}
There is a bijection between:
\begin{itemize}
    \item The equivalence classes $[h]$ of harmonic functions on $\IR^2 \backslash \{O\}$, where $O$ is the origin, given by:
    $$h_1 \sim h_2 ~~~~ \text{if and only if} ~~~ h_1 - h_2 \rightarrow C ~~ \text{when} ~~ |z| \rightarrow +\infty, $$
    where $C$ is a constant.
    \item $\mathbb{Z}_2$ harmonic functions $f$ on $X$.
\end{itemize}
The bijection is given by sending $f$ to $U(f)$.

\end{theorem}
We have already shown that the map $f \rightarrow U(f)$ is injective. So, the task is to prove that it is also surjective.\\

We need a lemma first:
\begin{lemma}
Fix a path from $z_1$ to each $z_{j}$ on the double branched cover $\Sigma$ of $\IC$ so that the following expression makes sense: $$v_{jk} = \int_{z_1}^{z_j} \dfrac{z^k}{\sqrt{D(z)}} dz.$$
And let $\Vec{v}_k = \{v_{2k}, v_{3k}, \cdots, v_{2n,k}\}$. Then the following $2n - 1$ vectors $$\text{Re}(\Vec{v}_0), \text{Im}(\Vec{v}_0), \text{Re}(\Vec{v}_1), \text{Im}(\Vec{v}_1), \cdots, \text{Re}(\Vec{v}_{n-2}), \text{Im}(\Vec{v}_{n-2}), \text{Im}(\Vec{v}_{n-1})$$ are linearly independent over $\IR$.
\end{lemma}
\begin{proof}
Suppose $$a_0\text{Re}(\Vec{v}_0) + b_0\text{Im}(\Vec{v}_0)+ \cdots+ a_{n-2}\text{Re}(\Vec{v}_{n-2})+b_{n-2} \text{Im}(\Vec{v}_{n-2})+b_{n-1} \text{Im}(\Vec{v}_{n-1}) = 0.$$

Let $$c_0 = a_0 - ib_0, c_1 = a_1 - ib_1, \cdots, c_{n-2} = a_{n-2} - ib_{n-2}, c_{n-1} = -ib_{n-1}. $$

And let $u(z) = c_0 + c_1 z + \cdots + c_{n-2}z^{n-2} + c_{n-1}z^{n-1}$.\\

Then 
$$\text{Re}(\int_{z_1}^{z_j} \dfrac{z^k}{\sqrt{D(z)}} dz) = 0, ~~~~ \text{for all}~~ j = 2, 3, \cdots, 2n. $$

So by theorem \ref{theorem of equivalence 1},
$$f = \text{Re}(\int_{z_1}^{z}\dfrac{u(z)}{\sqrt{D(z)}}) dz $$
defines a $\mathbb{Z}_2$ harmonic function.\\

Note that 
$$ \dfrac{u(z)}{\sqrt{D(z)}}  = - \dfrac{ib_{n-1}}{z} + O(|z|^{-2}) $$

So
$$f = \text{Re}(-ib_{n-1} \ln z) + C + O(|z|^{-1}), $$
where $C$ is a constant. But in order for $\text{Re}(-ib_{n-1} \ln z)$ to be well defined, $b_{n-1}$ must be $0$. So $f = C + O(|z|^{-1})$. So $U(f) = 0$. Hence $f = 0$, which implies $u(z) = 0$.
\end{proof}

Now we prove that $f \rightarrow U(f)$ is surjective.\\

Choose any $$U(f)  = c_0 \ln \rho + \text{Re}(\sum\limits_{n=1}^{+\infty} c_n \rho^n e^{in\theta}).$$

Consider the following formal Laurent series  $$\sqrt{D(z)} (\dfrac{c_0}{z} + \sum\limits_{n=1}^{+\infty} n c_n z^{n-1}) = \sum\limits_{k = -\infty}^{+\infty} C_kz^k.$$ 
Let $$u_0(z) = \sum\limits_{k=0}^{+\infty} C_k z^k.$$

Then $u_0(z)$ is an entire function in $\IC$ such that 
$$|dU(f) - \text{Re}(d(\dfrac{u_0(z)}{\sqrt{D(z)}}))| = O(\dfrac{1}{|z|^{n}}) ~~~~ \text{as} ~~ |z| \rightarrow +\infty. $$

By the previous lemma, there exist real numbers $a_0, b_0, a_1, b_1, \cdots a_{n-2}, b_{n-2}, b_{n-1}$ and a corresponding polynomial
$$u_1(z) = (a_0 - ib_0) + (z_1 - ib_1)z + \cdots + (a_{n-2} - ib_{n-2})z^{n-2} - ib_{n-1}z^{n-1}, $$

such that 

$$\text{Re}(\int_{z_1}^{z_j} \dfrac{u_0(z)}{\sqrt{D(z)}} dz) = \text{Re}(\int_{z_1}^{z_j} \dfrac{u_1(z)}{\sqrt{D(z)}} dz), $$
for each $j = 2, 3, \cdots, 2n$.\\

Thus $$f = \Re(\int_{z_1}^z \dfrac{(u_0(z) - u_1(z))}{\sqrt{D(z)}} dz) $$
defines a $\mathbb{Z}_2$ harmonic function.\\

Since $$df =  \text{Re}(\dfrac{u_0(z) - u_1(z)}{\sqrt{D(z)}} dz) = dU(f) + \text{Re}(\dfrac{ib_{n-1}}{z}) + O(|z|^{-1}) ~~ \text{when} ~~|z| \rightarrow +\infty,$$
so
$$f = U(f) + C + O(|z|^{-1}) ~~~ \text{when} ~~ |z| \rightarrow +\infty, $$
where $C$ is a constant. This implies that the unbounded part of $f$ is indeed the proscribed $U(f)$ and that the proof of theorem \ref{theorem of equivalence 2} is finished.\\

Note that the combination of theorem \ref{theorem of equivalence 1} and theorem \ref{theorem of equivalence 2} is just theorem \ref{theorem of two dimensional correspondence} stated in the introduction.\\

The remainder of this section concerns the case of the higher dimensions.\\

Recall that now the ambient space is $X = \IR^n \backslash K$, where $n \geq 3$ and $K$ is a compact smooth manifold of dimension $n-2$ embedded smoothly into $\IR^n$. Let $\mathcal{L}$ be the unique real bundle in $X$ with monodromy $-1$ around $K$. We can similarly define the $\mathbb{Z}_2$ harmonic functions $f$ in $X$ with monodromy $-1$ around $K$. \\

Suppose $f$ is a $\mathbb{Z}_2$ harmonic function in $U\backslash K$ with monodromy $-1$ around $K$, where $K$ is a smooth compactly embedded sub-manifold of codimension $2$ in $\IR^n$, and $U$ is a bounded open subset of $\IR^n$ that contains $K$. Suppose that the $L^2$ norm of $|\nabla f|$ is finite on $U$.\\

In addition, choose a large enough number $R_0$ and an open ball $B_{R_0}$ of radius $R_0$ centered at the origin such that $K \subset B_{R_0}$. Fix a trivialization of $\mathcal{L}$ in $\IR^n \backslash B_{R_0}$. Then $f$ is viewed as an ordinary function on $\IR^n \backslash B_{R_0}$. \\

\begin{lemma}
The function $|f|$ is a subharmonic function on $U \backslash K$. That is to say, for any ball $B$ centered at $P$ in $U$, where $P \notin K$,
$$|f(P)| \leq \dfrac{1}{\text{Vol}(B)} \int_{B \backslash K} |f|d^n \mathbf{x}.$$
Note that we allow $B$ to have intersection with $K$.
\end{lemma}

\begin{proof}

If $B \subset (U \backslash K)$, then $f$ is an ordinary harmonic function in $B$ up to the sign. So, the inequality is true by the mean value property.\\

In general, suppose $B_r$ is the ball of radius $r$ centered at $P$. Since $P \notin K$, the inequality is true when $r$ is small such that $B_r \cap K = \emptyset$.\\

Let

$$\phi(r) = \dfrac{1}{\text{Vol}(\partial B_r)} \int_{\partial B_r\backslash K}|f| d^n\mathbf{x}.$$

$$\text{Vol}(\partial B_r) \cdot \phi'(r) = \lim\limits_{\epsilon \rightarrow 0} (\int_{B_r\backslash K_{\epsilon}} \Delta(|f|) d^n \mathbf{x} + \int_{\partial K_{\epsilon}\cap B_r} \dfrac{\partial f}{\partial \vec{n}} d\sigma) \geq \liminf\limits_{\epsilon \rightarrow 0} \int_{\partial K_{\epsilon}\cap B_r} \dfrac{\partial f}{\partial \vec{n}} d\sigma. $$

Here

$$ |\int_{\partial K_{\epsilon}\cap B_r} \dfrac{\partial f}{\partial \vec{n}} d\sigma| \leq \int_{\partial K_{\epsilon}} |\nabla f| d\sigma \leq \sqrt{\text{Vol}_{n-1}(\partial K_{\epsilon})} (\int_{\partial K_{\epsilon}}|\nabla f|^2 d\sigma)^{\frac{1}{2}} \leq C_1(\int_{\partial K_{\epsilon}} |\nabla f|^2 \epsilon d\sigma)^{\frac{1}{2}}, $$

where $C_1$ is a constant that doesn't depend on $\epsilon$.\\

On the other hand, fix any small enough $\delta$, there exists a constant $C$ which doesn't depend on $\delta$, such that

$$\int_{K_{2\delta} \backslash K_{\delta}} |\nabla f|^2 d^n \mathbf{x} \geq C \int_{\delta}^{2\delta} \int_{\partial K_{\rho}} |\nabla f|^2 d\sigma d\rho \geq \dfrac{C}{2} \inf\limits_{\delta < \epsilon < 2\delta} \int_{\partial K_{\epsilon}} |\nabla f|^2 \epsilon d\sigma . $$

Note that $|\nabla f|$ has locally finite $L^2$ norm. The above inequality implies that

$$\liminf\limits_{\epsilon \rightarrow 0} \int_{\partial K_{\epsilon}} |\nabla f|^2 \epsilon d\sigma \leq \liminf\limits_{\delta \rightarrow 0}\int_{K_{2\delta} \backslash K_{\delta}} |\nabla f|^2 d^n \mathbf{x} = 0. $$

So

$$ \liminf\limits_{\epsilon \rightarrow 0} \int_{\partial K_{\epsilon}\cap B_r} \dfrac{\partial f}{\partial \vec{n}} d\sigma = 0.$$

So $\phi'(r) \geq 0$. And $|f(P)| \leq \phi(r)$ for any $r > 0$ such that $B_r \subset U$.
\end{proof}

\begin{corollary}
If we extend $|f|$ by setting $|f| = 0$ on $K$, then $|f|$ is a sub-harmonic function on the entire $U$.
\end{corollary}

Note that we haven't proved that the extension is continuous yet.

\begin{lemma}
Suppose $K_{\epsilon}$ is a small tubular neighborhood of $K$ which is a $D(\epsilon)$ bundle over $K$, where $D(\epsilon)$ is a disk of radius $\epsilon$ whose local polar coordinates are $(\rho, \theta)$. (Note that $\theta$ may not be well-defined globally along $K$). For a small contractible relatively open subset $U$ of $K$, let $U_{\epsilon}$ be the $\epsilon$ tubular neighborhood of $U$ (which is roughly $U \times D(\epsilon)$ with minor modifications on the metric). Then there exists a constant $C$ which doesn't depend on $\epsilon$, but may depend on $U$, such that

$$\dfrac{1}{\text{Vol}(U_{\epsilon})} \int_{U_{\epsilon}} |f| d^n \mathbf{x} \leq  C || \nabla f||_{L^2(U_{\epsilon})}. $$
Here $\text{Vol}(U_{\epsilon})$ volume of $U_{\epsilon}$.\\

In particular
$$\lim\limits_{\epsilon \rightarrow 0}\dfrac{1}{\text{Vol}(U_{\epsilon})} \int_{U_{\epsilon}} |f| d^n \mathbf{x} = 0.$$
\end{lemma}

\begin{proof}

Because $f$ has modronomy $-1$ around $K$, we know that for any $\theta_0$,
$$2|f| \leq \int_0^{2\pi}|\partial_{\theta} f| d\theta.$$
$$\dfrac{1}{2\pi}\int_0^{2\pi}|f| d\theta \leq \dfrac{1}{2} \int_0^{2\pi}|\partial_{\theta} f|d\theta \leq C \int_0^{2\pi} |\nabla f| \rho d\theta, $$
where $C$ is a constant. So there is a constant $C'$ such that

$$ \int_{U_{\epsilon}} |f| d^n\mathbf{x} \leq C' \int_{U_{\epsilon}} |\nabla f| \rho d^n \mathbf{x} \leq C' \epsilon (\sqrt{\text{Vol}(U_{\epsilon})})||\nabla f||_{L^2(U_{\epsilon})}.$$
And there is another constant $C_1$ such that
$$C' \epsilon \leq C_1 \sqrt{\text{Vol}(U_{\epsilon})}.$$
So
$$\dfrac{1}{\text{Vol}(U_{\epsilon})} \int_{U_{\epsilon}} |f|d^n \mathbf{x} \leq C_1 || \nabla f||_{L^2(U_{\epsilon})}. $$

The final assertion in the lemma is because $$\lim\limits_{\epsilon \rightarrow 0} ||\nabla f||_{L^2(U_{\epsilon})} = 0. $$

\end{proof}

\begin{lemma}
Let $C(\delta)$ be a subset of $K$ such that for each $P \in C(\delta)$, there exists a ball $B$ centered at $P$ such that
$$\dfrac{1}{\text{Vol(B)}} \int_{B}|f| d^n\mathbf{x} \leq \delta. $$
Then $C(\delta) = K$.
\end{lemma}

\begin{proof}
Clearly $C(\delta)$ is a closed subset of $K$. It suffices to prove that $C(\delta)$ is dense in $K$.\\

For any small contractible relatively open subset $U$ of $K$, we know from the previous lemma that

$$\dfrac{1}{\text{Vol}(U_{\epsilon})} \int_{U_{\epsilon}} |f| d^n \mathbf{x} \leq  C || \nabla f||_{L^2(U_{\epsilon})}. $$

But $$\lim\limits_{\epsilon \rightarrow 0} ||\nabla f||_{L^2(U_{\epsilon})} = 0. $$ 
So the ``average value" of $|f|$ on $U_{\epsilon}$ approaches $0$ as $\epsilon$ goes to $0$. This implies that there exists at least one point $P \in U$ such that $P$ is also in $C(\delta)$. Since this is true for any small contractible relative open subset of $K$, we know that $C(\delta)$ is dense.
\end{proof}

\begin{theorem}\label{high dimension |f| = 0 on K thoerem}
The function $|f|$ can be extended continuously in all $U$ setting $|f| = 0$ in $K$. This extended version of $|f|$ is a subharmonic function in $U$.
\end{theorem}

\begin{proof}
We have already proved that the extended version of $|f|$ is subharmonic. So, it suffices to prove that it is continuous at any point $P$ in $K$.\\

From the previous lemma, for any small enough $\delta > 0$, $P \in C(\delta)$. So there exists a ball $B$ centered at $P$ such that 

$$\dfrac{1}{\text{Vol}(B)} \int_B |f| d^n \mathbf{x} \leq \delta. $$

Let $B_0$ be the ball centered at $P$ whose radius is half of the radius of $B$. For any $P' \in B_0$, let $B(P')$ be the ball centered at $P'$ whose radius is also half of the radius of $B$. Then $B(P') \subset B$. And $\text{Vol}(B(P')) = \dfrac{1}{2^n} \text{Vol}(B)$. So

$$\dfrac{1}{\text{Vol}(B(P'))} \int_{B(P')}|f|d^n\mathbf{x} \leq \dfrac{1}{\text{Vol}(B(P'))} \int_{B}|f|d^n\mathbf{x} \leq 2^n \delta. $$

Since $|f|$ is subharmonic, for any $P' \in B_0$,
$$ |f(P')| \leq 2^n \delta.$$

Letting $\delta \rightarrow 0$ implies that

$$\limsup\limits_{P' \rightarrow P} |f(P')| = |f(P)| = 0.$$

Thus $|f|$ is continuous at $P$.

\end{proof}

Before we prove theorem \ref{theorem of higher dimensional analog}, we introduce an inspiring special case of theorem \ref{theorem of higher dimensional analog} due to Taubes:

\begin{corollary}\label{Taubes corollary}
There is a unique $\mathbb{Z}_2$ harmonic function on $\IR^n\backslash K$ such that
$$f - 1 \rightarrow 0 ~~~~ \text{when} ~~~ |\mathbf{x}| \rightarrow +\infty. $$
\end{corollary}
\textit{Taubes' proof of corollary \ref{Taubes corollary}}\\

Consider all smooth $\mathbb{Z}_2$ function that equals $1$ when $\mathbf{x}$ is large enough. Let the Dirichlet energy be

$$\mathcal{E}(f) = \int_{X}|df|^2 d\mathbf{x}. $$

Consider the following space:\\
 
Suppose $f$ is a smooth $\mathbb{Z}_2$ function such that $1 - |f| = 0$ far away from the origin and that $|f|$ extends continuously to $K$ by setting $|f| = 0$ on $K$. Then we define the $M$ ``norm" of $f$ as its Dirichlet energy:
$$ ||f||^2_M = \int_{X}|df|^2 d\mathbf{x}.$$
The space of completion of the above type of smooth $\mathbb{Z}_2$ functions under the $M$ norm is still denoted as $M$.

Note that if $f \in M$, then by Sobolev embedding, the $L^{\frac{2n}{n-2}}(X)$ norm of $1 - |f|$ is bounded up to a constant by $||f||_M$. As a corollary, $f$ also has finite $L^2(B)$ norm for any ball $B$ whose closure is in $X$.\\

Choose a sequence of smooth functions $f_n \in M$ whose Dirichlet energy approaches the infimum in $M$. One may assume, by possibly choosing a subsequence, that

\begin{itemize}
    \item $f_n$ weakly converges to $f$ in $M$.
    \item $1 - |f_n|$ weakly converges to $1 - |f|$ in $L^{\frac{2n}{n-2}}(X)$.
\end{itemize}

The second bullet implies $f \in M$. The first bullets implies
$$\liminf\limits_{n \rightarrow +\infty} ||f_n||_M = \inf\limits_{g\in M_C}||g||_M \geq ||f||_M.$$

Since $f \in M$, by the way that the sequence $f_n$ was constructed, $f$ achieves the minimal Dirichlet energy in $M$.\\

By the method of variation, under any local trivialization of $\mathcal{L}$, one sees that $f$ is a weak solution of the Laplacian equation on any small ball $B$ whose closure is in $X$. Since $f$ is also in $L^2(B)$, by a standard elliptic regularity argument, $f$ is harmonic on $B$. Thus $f$ is a $\mathbb{Z}_2$ harmonic function on the entire $X$.\\

Now we return to theorem \ref{theorem of higher dimensional analog}. We first need a lemma.

\begin{lemma}
Let $U$ be a connected, bounded open subset of $\IR^n$ such that $K \subset U$. Suppose that the boundary $\partial U$ is smooth. Suppose that a trivialization of $\mathcal{L}$ on $\partial U$ is given. Suppose $f$ is a smooth function on $\bar{U}$. Then there exists a unique $\mathbb{Z}_2$ harmonic function $u$ on $\bar{U} \backslash K$ 
(that is to say, harmonic on $U \backslash K$ and continuous on $\bar{U}\backslash K$) with monodromy $-1$ around $K$ (we just call it $\mathbb{Z}_2$ harmonic function when there is no ambiguity) such that
$$v = f ~~~ \text{on $\partial U$}. $$
\end{lemma}

\begin{proof}
Consider all smooth $\mathbb{Z}_2$ functions $v$ on $U \backslash K$ that equals $f$ on the boundary.\\

Consider the energy function
$$E_U(v) = \int_{U} |\nabla v|^2 d^n\mathbf{x}.$$

The energy function is bounded above by the $H^1(U)$ norm. By Poincare inequality, the $H^1(U)$ norm is also bounded above by a constant times $E_U(v)$ plus a constant which only depends on $U$ and $f$.
Thus we can choose a sequence of $\mathbb{Z}_2$ smooth functions $v_n$ on $\bar{U}\backslash K$  that converges to the infimum of the energy among all such $\mathbb{Z}_2$ functions and weakly converges to a weak solution of the Laplacian equation in $H^1(B_R)$, which by choosing a further sub-sequence strongly converges in $L^2(B_R)$. By a standard regularity theorem, this converging $\mathbb{Z}_2$ function as a weak solution of the Laplacian equation is a $\mathbb{Z}_2$ harmonic function.\\

Suppose there are two $\mathbb{Z}_2$ harmonic functions on $\bar{U}\backslash K$ which has the same boundary values on $\partial U$. Then their difference $\tilde{v}$ is a $\mathbb{Z}_2$ harmonic function which equals $0$ on $\partial U$. But if $\tilde{v} \neq 0$, then let $V_0$ be any non-empty open connected component of $U \backslash (K \cup Z)$, where $Z$ is the zero locus of $\tilde{v}$. Then such $V_0$ exists and $|\tilde{v}|$ is an ordinary non-zero harmonic function on $V_0$ whose boundary equals zero. This violates the maximal principle. Hence, the $\mathbb{Z}_2$ harmonic function is unique.
\end{proof}

\textit{Proof of theorem \ref{theorem of higher dimensional analog}}.
\begin{itemize}
    \item $h \Rightarrow f$
\end{itemize}

Let $f_R$ be the $\mathbb{Z}_2$ harmonic function on $B_R \backslash K$ defined in the previous lemma. We may extend $f_R$ continuously to the entire $X$ by setting $f_R = h$ outside of $B_R$, still denoted as $f_R$. Suppose $R_0 \ll R_1 \ll R \ll R_2$. Let $E_R$ be the short for $E_{B_R}$ which is the energy defined in the previous lemma.\\

Since $f_R = h$ outside of $B_R$, we have

$$E_{R_2}(f_R) - E_{R_2}(h) = E_{R}(f_R) - E_{R}(h). $$

Let $E(f_R) = \lim\limits_{R_2 \rightarrow +\infty} (E_{R_2}(f_R) - E_{R_2}(h)) = E_{R}(f_R) - E_R(h).$\\

Note that by the way $f_R$ was chosen, $f_R$ has the minimal $E_R$ energy among all $\mathbb{Z}_2$ functions in $B_R \backslash K$ equal to $u$ in $\partial B_R$. And $f_{R_1}$ is an example of such $\mathbb{Z}_2$ function. Thus, $E_R(f_R) \leq E_R(f_{R_1}).$\\

So $E(f_{R}) \leq E(f_{R_1})$. That is, $E(f_R)$ does not increase in $R$. Since it is bounded below by $0$, it must converge. We may assume $E = \lim\limits_{R \rightarrow +\infty} E(f_R)$.\\

Suppose $R_1$ is large enough such that $E(f_{R_1}) - E < \epsilon$. Then for any $R > R_1$, $E(f_{R_1}) - E(f_{R}) < \epsilon$. In particular, $|E_R(f_R) - E_R(f_{R_1})| < \epsilon$. On the other hand, $f_R$ has the minimal $E_R$ energy among all $\mathbb{Z}_2$ harmonic functions on $B_R \backslash K$ that is equal to $u$ on $\partial B_R$. And $\dfrac{f_R + f_{R_1}}{2}$ is one of such kind of $\mathbb{Z}_2$ functions. So
$$E_R(f_R + f_{R_1}) = 4 E_R(
\dfrac{f_R + f_{R_1}}{2}) \geq 4 E_R(f_R). $$

And

$$E_R(f_R - f_{R_1}) = 2(E_R(f_R)) + E_R(f_{R_1})) - E_R(f_R + f_{R_1}) < 2\epsilon. $$

Note that $f_R - f_{R_1}$ has bounded support in $\IR^n \backslash K$. By Sobolev embedding, there exists a constant $C$ which does not depend on $R$ such that
$$||f_R -f_{R_1} ||_{L^{\frac{2n}{n-2}}(X)} \leq C \epsilon.$$

This implies that $f_{R} - f_{R_1}$ is a Cauchy sequence (as $R$ goes to infinity) in $L^{\frac{2n}{n-2}}(X)$. So, it converges to some $\mathbb{Z}_2$ function $f - f_{R_1}$ as $R \rightarrow +\infty$. This function $f$ is locally a weak solution of the Laplacian equation and locally in $L^2$, so it is a $\mathbb{Z}_2$ harmonic function.\\

On the other hand,

$$||f_R - h||_{L^{\frac{2n}{n-2}}(\IR^n \backslash B_{R_1})} \leq C\epsilon.$$

So $$||f - h||_{L^{\frac{2n}{n-2}}(\IR^n \backslash B_{R_1})} \leq 2C\epsilon.$$
So for any ball $B$ of radius $2$ in $\IR^n \backslash B_{R_1}$,
$$||f - h||_{L^2(B)} \leq C_1 \epsilon, $$
where $C_1$ is a large constant.\\

But $f - h$ is a harmonic function in $B$. So, by a regularity argument, assuming that $B'$ is the ball of radius $1$ with the same center as $B$, we have another constant $C_2$ such that
$$||f - h||_{L^{\infty}(B')} \leq C_2\epsilon. $$

Note that $C_2$ does not depend on the particular ball $B$ chosen. So, in fact, for some $ R_2 \gg R_1$, we have

$$||f - h||_{L^{\infty} (\IR^n\backslash B_{R_2})} \leq C_2 \epsilon. $$

Letting $\epsilon \rightarrow 0$ implies that $f$ converges to $h$ uniformly as $R \rightarrow +\infty$. 

\begin{itemize}
    \item $f \Rightarrow h$
\end{itemize}

We may use the same argument as in the $h \Rightarrow f$ case. However, we introduce a different argument here.\\

In the spherical coordinates, the Laplacian equation is written as
$$\Delta f = \partial^2_rf + \dfrac{(n-1)}{r}\partial_rf + \dfrac{1}{r^2}\Delta_{S^{n-1}}f = 0.$$

And when $R$ is large enough, $$f = A_0 + B_0R^{-(n-2)} + \sum\limits_{m = 1}^{+\infty} A_m R^m\Phi_{m} + \sum\limits_{m=1}^{+\infty} B_mR^{-(m+n-2)}\Psi_{m},$$

where $\Phi_m$ and $\Psi_m$ are spherical harmonics with eigenvalue $(m+n-2)m$, $A_m$ and $B_m$ are constants.\\

Let 

$$h = A_0 + \sum\limits_{m = 1}^{+\infty} A_m R^m\Phi_{m}.$$

Then $h$ is the harmonic function on $\IR^n$ that we want.

\begin{remark}
Like in the case $\IR^2$, one sees from the proof that the $\mathbb{Z}_2$ function $u$ depends on $h$ and $K$ continuously in a suitable sense.\\
\end{remark}

Now we prove the corollary \ref{Corollary of higher dimensional case}. We may assume that our $S^1$ embeds in $\mathbb{R}^3$ as follows: 

$$S^1 = \{( \cos s,  \sin s, 0) ~ | ~ s \in [0, 2\pi)\}. $$
Clearly there is an $S^1$ action of $\mathbb{R}^3$ that restricts to a rotation on the embedded $S^1$.\\

We may also choose local coordinates $(s, \rho, \theta)$ in a neighborhood of $S^1$ to represent the point: $((\cos s)\cdot(1 + \rho \cos\theta), (\sin s) \cdot (1 + \rho \cos\theta), \rho\sin\theta)$.

\begin{lemma}
    There is a sequence of linearly independent harmonic functions in $\mathbb{R}^3$ that are invariant under the $S^1$ action. 
\end{lemma}

\begin{proof}
    In fact, these harmonic functions can be chosen to have polynomial growth. Otherwise, there are even uncountable such harmonic functions.\\
    
Examples are: $f = |x|^{l(l+1)}P_l(\cos s)$, where $|x| = \sqrt{x_1^2 + x_2^2 + x_3^2}$ is the distance to the origin, and $P_l$ are Lendgre polynomials. It is well known that $P_l(\cos s)$ are spherical harmonics, and hence $f$ are the harmonic functions as desired.
    
\end{proof}

By theorem \ref{theorem of higher dimensional analog}, we immediately have the following corollary:

\begin{corollary}\label{corollary of infinite f}
    There is a sequence of linearly independent $\mathbb{Z}_2$ harmonic functions in $\mathbb{R}^2 \backslash S^1$ with monodromy $-1$ around $S^1$ that are invariant under the $S^1$ action.
\end{corollary}

Here is a sketch of the proof of corollary \ref{Corollary of higher dimensional case}:

\begin{proof}In the local coordinates,
    $$\Delta f = 0 \iff  \partial_{\rho}(\rho(1 + \rho \cos\theta)^{-1} \partial_{\rho}f) + \partial_{\theta}(\rho^{-1}(1 + \rho \cos\theta)\partial_{\theta}f) + \partial_{s}(\rho(1 + \rho\cos\theta)\partial_s f) = 0.$$
Suppose $f$ is invariant under $S^1$ action, that is to say, independent with $s$. Then $\partial_sf = 0$ and

$$\partial_{\rho}^2f + \dfrac{1}{\rho}\partial_{\rho}f + \dfrac{1}{\rho^2}\partial_{\theta}^2f + (\text{higher order terms}) = 0.$$

We can construct the Fourier series of $f$ as in the proof of lemma \ref{local neibourhood lemma}.\\
Suppose
$$f = \text{Re}(\sum\limits_{m=0}^{+\infty} a_m(\rho)e^{\frac{(2m+1)i\theta}{2}}),$$
where $a_m(\rho)$  are complex-valued functions in $\rho$.

$$a_m''(\rho) + \dfrac{1}{\rho}a_m'(\rho) - \dfrac{(2m+1)^2}{4\rho^2}a_m(\rho) + (\text{higher order terms}) = 0. $$

From corollary \ref{corollary of infinite f}, there are infinitely many choices of linearly independent $\mathbb{Z}_2$ harmonic functions $f$ that clearly satisfy $|f| = O(\rho^{-1})$. We use induction on the power.\\

By induction, suppose that there are infinitely many choices of linearly independent $f$ such that $f = O(\rho^{m})$ as $\rho \rightarrow 0$. All such $f$ form an infinite-dimensional vector space $V_m$. For each $f \in V_m$, we have $a_k(\rho) = O(\rho^{m})$ for all $k$. From the standard regular singular point theory in ODE, we know that $a_k(\rho) = O(\rho^{m+1})$ for all $k \neq m$ and that $a_m(\rho) = A \rho^{m + \frac{1}{2}} + O(\rho^{m+1})$. The map from $V_m$ to $\mathbb{R}$ sending $f$ to $A$ has an infinite-dimensional kernel. Clearly, any $\mathbb{Z}_2$ harmonic function in the kernel behaves like $O(\rho^{m+1})$ as $\rho \rightarrow 0$.
\end{proof}

\begin{remark}
    Strictly speaking, in the above argument we need to check that all the higher-order remainder terms have all derivatives of the correct order as $\rho \rightarrow 0$. Moreover, this boundary behavior is uniform for different Fourier components $a_m(\rho)$. However, these are all easy to check and are omitted here. In fact, there are many much stronger conclusions about what a $\mathbb{Z}_2$ harmonic function is like near its singular points (as a polyhomogeneous expansion) in \cite{SiqiHe2022ExistenceSymmetry},\cite{Donaldson2021DeformationsFunctions}, \cite{Parker2023Deformations3-Manifolds}, etc. We refer interested readers to these references.
\end{remark}

\section{The polynomial Pell equation}\label{Section: Pell}

In this section we prove theorem \ref{theorem of Pell}, theorem \ref{theorem of Pell property}, and describe the details of the practical method to check whether a given $D(z)$ is Pellian or not. We work in the two-dimensional case. We assume that $K$ and $D(z)$ are given.\\

By theorem \ref{theorem of two dimensional correspondence}, there exists a unique $\mathbb{Z}_2$ harmonic function, denoted as $G_D(z)$ (or $G(z)$ for short), such that the unbounded part at infinity $U(G)$ has only a term $\ln |z|$. By corollary \ref{unique corollary with bounded zero locus}, it is also the unique $\mathbb{Z}_2$ harmonic function with bounded zero locus up to a real-valued scalar multiplication. This is the second bullet of the theorem \ref{theorem of Pell}.

\begin{remark}
The reason we call it $G(z)$ is due to the following observation of Taubes: if one uses a conformal map to send $\IC$ to $\IC\IP^1 \backslash \{S\}$, where $S$ is the south pole (corresponding to the point of infinity), then a $\mathbb{Z}_2$ harmonic function on $\IC$ becomes a $\mathbb{Z}_2$ harmonic function on $\IC\IP^1 \backslash \{S\}$. And $G(z)$ becomes the $\mathbb{Z}_2$ version of the Green function on $\IC \IP^1$ evaluated at $S$. More details of this Green function can be found implicitly in \cite{Taubes2021ExamplesThree} \cite{CliffordTopologicalS2} or \cite{Donaldson2021DeformationsFunctions}.
\end{remark}

\begin{theorem}
When $f(z) = G(z)$, the entire function corresponding to $u(z)$ in theorem \ref{entire function} is a polynomial of degree $n-1$, and vice versa (up to a multiplication of a constant of real value). We call this polynomial  $U(z)$. Then $U(z)$ is uniquely determined by $D(z)$ (up to a multiplication by a constant of real value as always).
\end{theorem}

\begin{proof}
Once the first statement is proved, the uniqueness is a consequence of theorem \ref{uniqueness theorem} and corollary \ref{unique corollary with bounded zero locus}. The first statement (to be more precise, the first sentence in the above statement of the theorem) is a special case of the upcoming theorem \ref{theorem of when u is a polynomial} whose proof is deferred until there.
\end{proof}

\begin{definition}\label{Definition of U}
For each even degree polynomial $D(z)$ (may not have distinct roots), we define a polynomial $U_D(z)$ (up to a multiplication by a real number) in the following way:

\begin{itemize}
    \item If all roots of $D(z)$ are distinct, then $U_D(z)$ is the unique polynomial $U(z)$ from the previous theorem.
    \item Suppose $D(z) = D_1(z)^2D_0(z)$, such that all roots of $D_0(z)$ are distinct. Then  $U_D(z) = D_1(z) U_{D_0}(z)$.
\end{itemize}
If there is no ambiguity, we may omit the subscript $D$ and simply write it as $U(z)$.
\end{definition}

\begin{theorem}\label{theorem of when u is a polynomial}
Suppose $f$ is a $\mathbb{Z}_2$ harmonic function. Suppose the unbounded part of $f$ at infinity is
$$U(f) = c_0 \ln \rho + Re(\sum\limits_{n=1}^{+\infty} c_n\rho^ne^{in\theta}). $$
Then $u(z)$ is a polynomial of degree $k + n - 1$ if and only if $k$ is the largest integer such that $c_k$ is not zero. 

\end{theorem}

\begin{proof}
In fact, when $\rho$ is large enough, we have expression (\ref{Fourier series of f at infinity})
   $$f =  (c_0 \ln \rho + \text{Re}(\sum\limits_{n=1}^{+\infty} c_n \rho^n e^{in\theta})) + (b_0 + \text{Re}(\sum\limits_{n=1}^{+\infty} \dfrac{b_n}{\rho^n} e^{in \theta}))$$ $$ = b_0 +  \text{Re}(c_0 \ln z + \sum\limits_{n=1}^{+\infty} (c_nz^n + b_n z^{-n})).$$
So

$$df = \text{Re}(\dfrac{u(z )dz}{\sqrt{D(z)}}) = \text{Re}(\dfrac{c_0}{z}dz + \sum\limits_{n=1}^{+\infty} (nc_nz^{n-1} - nb_nz^{-n-1})dz). $$

So
$$u(z) = \sqrt{D(z)}(\dfrac{c_0}{z} + \sum\limits_{n=1}^{+\infty}(nc_nz^{n-1} - nb_nz^{-n-1})).$$

Therefore, $u(z)$ has a pole of order $n + k-1$ at infinity if and only if $k$ is the largest integer such that $c_k$ is non-zero. But $u(z)$ is an entire function. So, it has a pole of order $n+k-1$ at infinity if and only if it is a polynomial of degree $n+k-1$.
\end{proof}

Note that there is a corollary: the leading coefficient of $U(z)$ is real.\\

\textbf{Remark:} The $\mathbb{Z}_2$ harmonic function $f$ depends on $[h]$ and $D(z)$ continuously in a suitable sense. In fact, in a suitable sense (for example, the $C^0(B)$ topology, where $B$ is a large enough disk that contains all the roots of $D(z)$), the following expressions for $D(z)$:
$$\int_{z_1}^{z_j} \dfrac{z^k}{\sqrt{D(z)}}dz $$
and the following expressions for $[h]$
$$\int_{z_1}^{z_j}\dfrac{u_0(z)}{\sqrt{D(z)}}dz$$
are all continuous, where $u_0(z)$ is defined in the proof of theorem \ref{theorem of higher dimensional analog}.

Recall that suppose $D(z)$ is a polynomial with complex coefficients. The following equation is called the polynomial Pell equation, or the Pell equation for short:
$$ p(z)^2 - q(z)^2 D (z) = 1, $$
where $p(z)$ and $q(z)$ are unknown polynomials with complex coefficients.\\

If the Pell equation for $D(z)$ has a nontrivial solution (that is, $q(z) \neq 0$), then $D(z)$ is called Pellian. Ohterwise, it is called non-Pellian.\\

One obvious observation is that, $D(z)$ is non-Pellian when the degree is odd. And when the degree of $D(z)$ is $2$, it is Pellian if and only if it is square free. It is believed that when the degree of $D(z)$ is at least $4$, then Pellian polynomials are extremely rare. There are many studies on this topic; see, for example, \cite{Dubickas2004TheEquation} \cite{Kalaydzhieva2020OnFields} \cite{Nathanson1976PolynomialEquations}.
For the remainder of this section, we assume that the degree of $D(z)$ is even.\\

The following proposition and its corollaries are well known and are very easy to prove, and we omit the proofs here: 

\begin{proposition}
If $(p_1(z), q_1(z))$ is a solution to the Pell equation for $D(z)$ with minimal degree. Let 
$$p_k(z) + q_k(z) \sqrt{D(z)} = (p_1(z) + q_1(z) \sqrt{D(z)})^k. $$
Then $(\pm p_k, \pm q_k)$ are all the solutions to the Pell equation for $D(z)$.
\end{proposition}

\begin{corollary}
Suppose $(p, q)$ is a solution to the Pell equation for $D(z)$. Let $\tilde{U}(z) = \dfrac{p'(z)}{q(z)}$. Then different solutions will give the same $\tilde{U}(z)$ up to a multiplication by a real-valued constant. 
\end{corollary}

\begin{corollary}
If $D(z)$ is Pellian and if $(p(z), q(z))$ is a solution to the Pell equation. Let $\tilde{U}(z) = \dfrac{p'(z)}{q(z)}$. Then $\tilde{U}(z)$ is a polynomial and the following Abel integration has an algebraic expression:
$$\int \dfrac{\tilde{U}(z)}{\sqrt{D(z)}} dz = \ln (p(z) + q(z) \sqrt{D(z)}). $$
\end{corollary}

Thus in order to prove theorem \ref{theorem of Pell property}, the only task is to verify that $\tilde{U}(z)$ given above is the same as our $U_D(z)$ up to a scalar multiplicity by a real-valued constant. Here is the \textbf{proof of theorem \ref{theorem of Pell property}}:

\begin{proof}
We only need to check that $\text{Re}(\ln(p(z) + q(z)\sqrt{D(z)}))$ defines a $\mathbb{Z}_2$ harmonic function. Then by uniqueness given by theorem \ref{theorem of two dimensional correspondence}, $U_D(z)$ must be the same as $\tilde{U}(z)$.\\

First of all, $$(p(z) + q(z)\sqrt{D(z)})(p(z) - q(z) \sqrt{D(z)}) = 1.$$
So $$p(z) + q(z)\sqrt{D(z)} \neq 0 ~~~\text{for all}~~~ z \in \IC.$$

Thus $\ln(p(z) + q(z)\sqrt{D(z)})$ has a well-defined real part on the entire $\IC$ except a sign ambiguity for $\sqrt{D(z)}$. But a different choice of the sign of $\sqrt{D(z)}$ also changes only the sign of $\text{Re}(\ln(p(z) + q(z) \sqrt{D(z)}))$. So, this sign ambiguity corresponds to the correct monodromy for a $\mathbb{Z}_2$ function. Thus, $\text{Re}(\ln (p(z) + q(z)\sqrt{D(z)}))$ is a well-defined $\mathbb{Z}_2$ harmonic function on $X$.

\end{proof}

Note that this above theorem partially explains why we choose to have the second bullet of the definition \ref{Definition of U}. The author wonders whether $U(z)$ can be computed from $D(z)$ in a purely algebraic way (that is, not doing the integrals).\\

When the coefficients of $D$ are real numbers, then the coefficients of $U$ are also real numbers. In fact, it is not hard to check that $U_{\bar{D}} = \bar{U}_D$, where $\bar{D}$ means the polynomial whose coefficients are conjugates of the coefficients of $D$.\\

Finally, inspired by theorem \ref{theorem of two dimensional correspondence}, we have the following general method to check whether a polynomial $D(z)$ is Pellian or not.

\begin{theorem}
Suppose $D(z) \in \IC[z]$ is a monic polynomial of even degree with distinct roots. Then $D(z)$ is Pellian if and only if all the following integrals
    $$\int_{z_1}^{z_j} \dfrac{U(z)}{\sqrt{D(z)}} dz $$
    are integer multiplications of $\pi i$.
\end{theorem}

    \begin{proof}
     If $D(z)$ is Pellian, suppose
     $$p(z)^2 - q(z)^2 D(z) = 1. $$
     Then $U(z) = \dfrac{p'(z)}{q(z)}$ and
     $$ \int_{z_1}^z \dfrac{U(z)}{\sqrt{D(z)}}dz = \ln(p(z) + q(z)\sqrt{D(z)}).$$
     But from Pell equation, $p(z_j) = \pm 1$ for any $j$, so
     $$\int_{z_1}^{z_j} \dfrac{U(z)}{\sqrt{D(z)}} dz $$
     is an integer multiplication of $\pi i$.\\
     
     The other direction: if all $$\int_{z_1}^{z_j} \dfrac{U(z)}{\sqrt{D(z)}} dz $$
     are integer multiplications of $\pi i$, then 
     $$\exp (\int_{z_i}^z \dfrac{U(z)}{\sqrt{D(z)}} dz) $$
     defines a holomorphic function on $\Sigma = \{(w, z) \in \IC^2 | w^2 = D(z)\}$ for the following reason:\\
     
     We know that the topology of $\Sigma$ is a surface of genus $n-1$ minus two points. The first homology of $\Sigma$ is generated by the following cycles:\\

 For each $z_j \neq z_1$, one goes from $z_1$ to $z_j$ in one branch and go back from $z_j$ to $z_1$ using the same path but in another branch. This loop on $\Sigma$ is denoted as $C_j (j \neq 1)$.\\

   We have assumed that the path integration of $\dfrac{U(z)dz}{\sqrt{D(z)}}$ is a multiplication of $2\pi i$ along any other loop $C_j$. So the exp of this integral is well-defined holomorphic function on $\Sigma$ without ambiguity.\\
     
     Finally, by GAGA (see for example, \cite{Neeman2007AlgebraicGeometry}),  a holomorphic function on $\Sigma$ with polynomial growth must be a polynomial in $w$ and $z$. Since $w^2 = D(z)$, we may write this polynomial as 
     $$p(z) + q(z)w = p(z) + q(z) \sqrt{D(z)}.$$
     
     And the unique $\mathbb{Z}_2$ harmonic function with bounded zero locus is $$G(z) = \ln |p(z) + q(z) \sqrt{D(z)}|. $$
     
Since $G(z)$ has monodromy $-1$ around any point $z = z_i$, a change of the sign of $\sqrt{D(z)}$ leads to a change of sign in $G(z)$ as well. We have
$$\ln|p(z) + q(z) \sqrt{D(z)}| = - \ln|p(z) - q(a) \sqrt{D(z)}|. $$
So $|p(z)^2 - q(z)^2 D(z)| = 1.$ We may assume $p(z)^2 - q(z)^2 D(z) = \epsilon$, where $\epsilon$ is a complex-valued constant with norm $1$. Then $(\epsilon^{-\frac{1}{2}}p(z), \epsilon^{-\frac{1}{2}}q(z))$ is a nontrivial solution to the Pell equation for $D(z)$.
    \end{proof}

\textbf{Remark:} Although we assumed that $D(z)$ does not have distinct roots here, it is standard to reduce the case more general to the distinct roots case because of the following well-known and easy-to-proof theorem (whose proof is thus omitted):

\begin{theorem}\label{old theorem 1}
If $U(z)$ has a common factor $(z - z_i)$ with $D(z)$. Let $k_1$ be the order of $(z - z_i)$ in $D(z)$ and $k_2$ be the order of $(z - z_i)$ in $U(z)$. Then $k_2 \geq k_1 - 1$. In addition,  $(z-z_i)^2 D(z)$ is Pellian if and only if $k_2 \geq k_1$. 
\end{theorem}

To finish this section, we write down a concrete example of a $\mathbb{Z}_2$ harmonic function. This example is due to Taubes, who told me in a private conversation. Exactly this example inspired me to think about the relationship between the polynomial Pell equation and $\mathbb{Z}_2$ harmonic functions.
\begin{example}
Let $D(z) = z^{2n}-1$, then
$$G(z) = \ln |z^n + \sqrt{z^{2n}-1}|. $$
\end{example}

\section{The zero locus}\label{Section: zero locus}

This short section studies the zero locus of the $\mathbb{Z}_2$ harmonic functions $f$ in $\mathbb{R}^2 \backslash K$ and its relationship to $u(z)$ in the correspondence given by theorem \ref{theorem of two dimensional correspondence}.\\

For an ordinary harmonic function in $\IR^2$, there is a famous theorem that states that the zero locus is a union of smooth curves such that, at any intersection points of the curves, there are only finitely many curves that intersect equiangularly. See, for example, the textbook \cite{Moore1930FoundationsTheory.} chapter X section 9.  We deduce an analog of this theorem for $\mathbb{Z}_2$ harmonic functions on $X$.

\begin{theorem}
Suppose $f$ is a $\mathbb{Z}_2$ harmonic function on $X$. Then the zero locus of $f$ is the union of locally finite smooth curves whose boundaries are empty or belong to $K$. At any intersection point of those smooth curves (including intersections at boundaries), they form an equiangular system. In addition to that, any subset of the zero locus of $f$ does not form a loop.
\end{theorem}

\begin{proof}
Here is the sketch. Suppose $P$ is a zero point of a $\mathbb{Z}_2$ harmonic function $f$.\\

If $P$ does not belong to $K$, then it is the same as the usual situation and a proof can be found in the textbook \cite{Moore1930FoundationsTheory.}. Otherwise,  $P$ belongs to $K$. Near $P$,
$$f = \text{Re}(\sum\limits_{n = 0}^{+\infty}  a_n z^{\frac{(2n+1)}{2}}) $$

Suppose $k \geq 0$ is the smallest number such that $a_k \neq 0$. We may assume $a_k$ is real, then 
$$f = a_k \rho^{\frac{(2k+1)}{2}} \cos(\frac{(2k+1)}{2}\theta) + \text{Re} (\sum\limits_{n=k+1}^{+\infty} a_n z^{\frac{(2n+1)}{2}}). $$

So $$f = a_k \rho^{\frac{(2k+1)}{2}} \cos(\frac{(2k+1)}{2}\theta) + O(\rho^{\frac{(2k+3)}{2}}). $$
$$df = a_k \rho^{\frac{(2k-1)}{2}}\cos(\frac{(2k+1)}{2}\theta) d\rho - a_k\rho^{\frac{(2k+1)}{2}}\sin(\frac{(2k-1)}{2} \theta)d\theta + O(\rho^{\frac{(2k+1)}{2}}). $$
The above two expressions, together with the implicit function theorem, are enough to show that in a neighborhood of $P$, the zero locus of $f$ are $2k+1$ smooth curves that intersect at their common boundary $P$ in an equiangular way.\\

That the zero locus does not form a closed loop is just a consequence of the maximal principle. And here is the argument:\\

If on the contrary, the zero locus $Z$ of $|f|$ has a closed loop as a subset. Then at least one connected component of $\IR^2 \backslash Z$ is bounded, denoted as $U$. Note that setting $f > 0$ defines a trivialization of $\mathcal{L}$ on $U$. And $f$ is an ordinary harmonic function on $U$ without zero points in it. In addition, $|f|$ can be extended as a continuous function to the closure of $U$ by setting $|f| = 0$ on the boundary. This implies that $f$ has a positive local maximum in $U$, which is impossible.
\end{proof}

\begin{definition}
Suppose $z_0$ is a zero point of a $\mathbb{Z}_2$ harmonic function $f$. Suppose that there are a non-zero number $a$ and a half integer $k$ such that $f = \text{Re}(a(z-z_0)^{k}) + O(|z-z_0|^{k+1})$ in a neighborhood of $z$, then $k$ is called the degree of $f$ at $z_0$. From the proof of the previous theorem, we know that the zero locus of $f$ near $k$ is an intersection of $2k$ curves that share a common boundary point at $z_0$ and intersect at $z_0$ equiangularly.
\end{definition}

\begin{theorem}
Suppose $z_0$ is a zero point of a $\mathbb{Z}_2$ harmonic function $f$. Suppose $u(z)$ is the entire function defined in section 1. If the degree of $f$ at $z_0$ is larger than $1$, then $z_0$ is a zero point of $u(z)$ with degree $\lfloor k - \dfrac{1}{2}\rfloor$ (the largest integer that is no greater than $k - \dfrac{1}{2}$).

\end{theorem}
\begin{proof}
If $z_0$ is a zero point of $f$ whose degree $k > 1$, then $z_0$ is also a zero point of $df + i*df$ whose degree is $k - 1$. Since $u(z) = \sqrt{D(z)} (df + i*df)$, $z_0$ is a zero point of $u(z)$ with degree $\lfloor k-\dfrac{1}{2} \rfloor$.
\end{proof}

Here are some easy examples.

\begin{example}
Suppose $D(z)$ is Pellian, and suppose $$p(z)^2 - q(z)^2 D(z) = 1.$$

Then the zero locus of $f$ is given by the set $$\{z |~ |p(z) + q(z)\sqrt{D(z)}| = 1\} .$$
This implies at a point in the zero locus,
$$\text{Im}(p(z)) = \text{Re}(q(z)\sqrt{D(z)}) = 0 ~~~ \text{and} ~~ (\text{Re}(p(z)))^2 + (\text{Im}(q(z)\sqrt{D(z)}))^2 = 1. $$
In particular, 
\begin{itemize}
    \item If $D(z)$ has only two roots and if they are distinct, then $$D(z) = (z - z_1)(z - z_2) = (z - k)^2 - m^2.$$ Then we can choose $p(z) = \dfrac{z-k}{m}$, $q(z) = \dfrac{1}{m}$ and the zero locus is the straight line segment connecting $z_1$ and $z_2$. In this situation, $U(z) = 1$.
    \item If $D(z) = p(z)^2 - 1$, and $q(z) = 1$, then the zero locus is 
    $$\{z | ~~ p(z) ~ \text{is real and} ~ |p(z)| \leq 1\} = p^{-1}([-1, 1]).$$
     An typical example of this kind is: $D(z) = z^{2n} - 1$.
\end{itemize}

\end{example}

\begin{example}

If $D(z)$ is a polynomial of degree $4$ with distinct roots and suppose $D(z) = (z - z_1)(z - z_2)(z - z_3)(z - z_4)$, then the degree of $U(z)$ is one. We may assume $U(z) = z - z_0$. (Note that $D(z)$ may not be Pellian in general.) Then there are two possibilities:
\begin{itemize}
    \item The zero locus is a combination of three curve segments and we may assume that they share a common end at $z_1$ and their other ends are $z_2$, $z_3$, $z_4$ respectively. They intersect equiangularly at $z_1$ and do not intersect at any other point. In this situation, $z_0 = z_1$. An typical example of this kind is: $D(z) = z^4 - z$.
    \item The zero locus is a combination of two curve segments and we may assume that one of them connects $z_1$ and $z_2$, the other one connects $z_3$ and $z_4$. If they intersect at a point in an orthogonal way, then this point must be $z_0$. Otherwise, they do not intersect.
\end{itemize}

The author is very curious how the four points $z_1, z_2, z_3$ and $z_4$ determine $z_0$ in general.
\end{example}

\bibliography{references.bib}
\bibliographystyle{ieeetr}

\end{document}